\documentclass[12pt,onesided]{amsart}
\RequirePackage[colorlinks,citecolor=blue,urlcolor=blue]{hyperref}
\usepackage{latexsym,amssymb,amsmath,amsfonts,amsthm}
\usepackage{amsthm,amsmath}
\usepackage{float}
\usepackage{enumerate}
\usepackage{epic}
\usepackage[margin=1.5in]{geometry}
\usepackage{bbm}
\usepackage{subfigure}
\usepackage{graphicx}
\usepackage[toc,page]{appendix}
\usepackage{multirow}
\usepackage{amsfonts}
\usepackage[all]{xy}
\usepackage{caption}
 \usepackage{geometry}                % See geometry.pdf to learn the layout options. There are lots.
\usepackage{graphicx}
\usepackage{amssymb}
\usepackage{epstopdf}
\usepackage{color}
\usepackage{bbm, dsfont}
\usepackage{hyperref}
\usepackage[utf8]{inputenc}
\usepackage{amssymb,amsthm,amsmath,amssymb,wrapfig,dsfont}
\usepackage[dvipsnames]{xcolor}
\definecolor{myred}{RGB}{251,154,133}
\definecolor{myblue}{RGB}{153,206,227}
\definecolor{mylightblue}{RGB}{0, 150, 255}
\definecolor{mygreen}{RGB}{32, 210, 64}
\definecolor{mygray}{RGB}{220, 220, 220}

\usepackage{tikz}
\usetikzlibrary{decorations.pathmorphing}
\tikzset{snake it/.style={decorate, decoration=snake}}
\usetikzlibrary{shapes.geometric,positioning,decorations.pathreplacing}

\newtheorem{theorem}{Theorem}
\newtheorem{definition}{Definition}
\newtheorem{lemma}{Lemma}[section]
\newtheorem{remark}{Remark}
\newtheorem{Proposition}{Proposition}
\newtheorem{corollary}{Corollary}

\newtheorem{counterexample}{Counterexample}

\def\beq{ \begin{equation} }
\def\eeq{ \end{equation} }

\def\square{\vcenter{\vbox{\hrule height .4pt
  \hbox{\vrule width .4pt height 5pt \kern 5pt
        \vrule width .4pt} \hrule height .4pt}}}

\def\RR{\mathbb{R}}
\def\ZZ{\mathbb{Z}}
\def\QQ{\mathbb{Q}}

\newcommand{\upspace}{\mathbb{H}}

\newcommand{\sharm}{\bar{\mathcal{H}}}

\begin{document}

\title{On sets of zero stationary harmonic measure}

\author{Eviatar B. Procaccia}
\address[Eviatar B. Procaccia \footnote{Research supported by NSF grant DMS-1812009.}]{Texas A\&M University}
\urladdr{www.math.tamu.edu/~procaccia}
\email{eviatarp@gmail.com}
 
\author{Yuan Zhang}
\address[Yuan Zhang \footnote{The manuscript of the first version of this paper was done when YZ was a visiting assistant professor at Texas A\&M University} \footnote{Research partially supported by Peking University grant 7101300248, 7101300249}]
{Peking University}
%\urladdr{http://www.math.tamu.edu/~yzhang1988/}
\email{zhangyuan@math.pku.edu.cn}

\maketitle

\begin{abstract}
In this paper, we prove that any subset with an appropriate sub-linear horizontal growth has a non-zero stationary harmonic measure. On the other hand, we also show any subset with super-linear horizontal growth will have a $0$ stationary harmonic measure at every point. This result is fundamental to any future study of stationary DLA. As an application we prove that any possible aggregation process with growth rates proportional to the stationary harmonic measure has non zero measure at all times. 
\end{abstract}

\section{Introduction}

In this paper, we present conditions for an infinite subset in the upper half planar lattice to have non-zero stationary harmonic measure. Stationary harmonic measure is first introduced in \cite{Stationary_Harmonic_Measure}, and plays a fundamental role in the study of diffusion limit aggregation (DLA) models on non-transitive graphs with absorbing boundary conditions. Roughly speaking, the stationary harmonic measure of a subset is the expected number of random walks hitting each of its points, when we drop infinitely many random walks from a horizontal line ``infinitely high" and stop once they first hit the subset or the $x-$axis. The stationary harmonic measure can be used to construct stationary DLA , which plays an equivalent role as the harmonic measure in $\ZZ^d$ in the construction of the DLA model (see \cite{harmonic_measure_1987,DLA_long,DLA_re-visit}). The analysis in this paper is crucial to the understanding of aggregation phenomenon under absorbing boundary conditions. The non rigorous message of this papers is that if an aggregation process grows too fast it might stop growing but if the growth is moderate it will keep on growing forever. 

For the precise discussions, we first set some notations defined in \cite{Stationary_Harmonic_Measure}. Let $$\upspace=\{(x,y)\in \ZZ^2, y\ge 0\}$$ be the upper half planar lattice (including $x$-axis), and $S_n, n\ge 0$ be a 2-dimensional simple random walk. For any $x\in \ZZ^2$, we will write 
$
x=(x_1,x_2)
$,
with $x_i$ denoting the $i$-th coordinate of $x$. Then define the horizontal line of height $n$
$$
L_n=\{(x,n), \ x\in \ZZ\}.
$$
For any subset $A\subset \ZZ^2$ abbreviate the first hitting time 
$$
\bar \tau_A=\min\{n\ge 0, \ S_n\in A \}
$$
and the first exit time
$$
\tau_A=\min\{n\ge 1, \ S_n\in A \}.
$$
For any subsets $A_1\subset A_2$ and $B$ and any $y\in \ZZ^2$, by definition one can easily check that 
\beq
\label{basic 1}
P_y\left(\tau_{A_1}<\tau_{B}\right)\le P_y\left(\tau_{A_2}<\tau_{B}\right), 
\eeq
and that 
\beq
\label{basic 2}
P_y\left(\tau_{B}<\tau_{A_2}\right)\le P_y\left(\tau_{B}<\tau_{A_1}\right). 
\eeq
Now we define the stationary harmonic measure on $\upspace$. For any $B\subset \upspace$, any edge $\vec e=x\to y$ with  $x\in B$, $y\in \upspace\setminus B$ and any $N$, we define 
\beq
\label{harmonic measure edge}
\sharm_{B, N}(\vec e)=\sum_{z\in L_N\setminus B} P_z\left(S_{\bar \tau_{B\cup L_0}}=x, S_{\bar \tau_{B\cup L_0}-1}=y\right).
\eeq

\begin{remark}
As a paper on the nearest neighbor aggregation process, \cite{Stationary_Harmonic_Measure} concentrate mostly on the case when $B$ or $B\cup L_0$ is connected. However, it is clear to see that the definition of $\sharm_{B, N}(\vec e)$ as well as the convergence in Proposition \ref{proposition_well_define} is not related to connectivity and thus hold for any $B$. 
\end{remark}

By definition, $\sharm_{B, N}(\vec e)>0$ only if $y\in \partial^{out}B$ and $|x-y|=1$. For all $x\in B$, we can also define  
\beq
\label{harmonic measure point old}
\sharm_{B, N}(x)=\sum_{\vec e \text{ starting from } x}\sharm_{B, N}(\vec e) =\sum_{z\in L_N\setminus B} P_z\left(S_{\bar \tau_{B\cup L_0}}=x\right).
\eeq
And for each point $y\in \partial^{out}B$, we can also define 
\beq
\label{harmonic measure point new}
\hat {\mathcal{H}}_{B, N}(y)=\sum_{\tiny\begin{aligned}\vec e &\text{ starting in $B$}\\ &\text{ ending at } y\end{aligned}}\sharm_{B, N}(\vec e) =\sum_{z\in L_N\setminus B} P_z\left(\tau_{B}\le \tau_{L_0} , S_{\bar \tau_{B\cup L_0}-1}=y\right).
\eeq
In \cite{Stationary_Harmonic_Measure} we prove that, 
\begin{Proposition}[Proposition 1 in \cite{Stationary_Harmonic_Measure}]
\label{proposition_well_define}
For any $B$ and $\vec e$ above, there is a finite $\sharm_{B}(\vec e)$ such that 
\beq
\lim_{N\to\infty} \sharm_{B, N}(\vec e)=\sharm_{B}(\vec e). 
\eeq
\end{Proposition}
And we call $\sharm_{B}(\vec e)$ the stationary harmonic measure of $\vec e$ with respect to $B$. We immediately have that the limits $\sharm_{B}(x)=\lim_{N\to\infty}\sharm_{B, N}(x)$ and $\hat{\mathcal{H}}_{B}(y)=\lim_{N\to\infty}\hat{\mathcal{H}}_{B, N}(y)$ also exists and we call them the stationary harmonic measure of $x$ and $y$ with respect to $B$.  

Note that the stationary harmonic measure is not a probability measure. When using the stationary harmonic measure as growth rate, or more precisely, letting $\hat{\mathcal{H}}_{B}(y)$ be the Poisson intensity the state at site $y$ changes from 0 to 1, we defined in \cite{Stationary_Harmonic_Measure} the (continuous time) DLA process in the upper half plane $\upspace$, starting from any finite initial configuration. For a finite subset $B$, it is shown in \cite[Theorem 3]{Stationary_Harmonic_Measure} that there must be an $x\in B$ such that $\sharm_B(x)>0$. This implies that the continuous DLA model will keep growing from any configuration.

Meanwhile, such treatment of using stationary harmonic measure as Poisson intensities rather than probability distribution also opens the possibility to study DLA from an {\bf infinite} initial configuration as an infinite interacting particle system. However, before possibly defining such an infinite growth model, one first has to ask which configuration in $\upspace$ can be ``habitable" for our aggregation. In fact, for infinite $B$, it is possible for $\sharm_B(\cdot)$ to be uniformly 0. Thus, for the possible DLA starting from such configuration, it will freeze forever without any growth.  The intuitive reason for such phenomena is that when $B$ is infinite, each point $x\in B$ may live in the shadow of other much higher points, which will block the random walk starting from ``infinity" to visit the former first.  In the following counterexample, we see that there can be a uniformly 0 harmonic measure even when the height of $B$ is finite for each $x-$coordinate. We encourage the reader to check the subset here has zero stationary harmonic measure before reading the proof of the main results. 
\begin{counterexample}
\label{counterexample_0}
Let 
$$
B^0=\bigcup_{n=-\infty}^\infty \{(n,k), \ k=0,1,\cdots, 2^{|n|}\}.
$$
Then $\sharm_{B^0}(x)=0$ for all $x\in B^0$. 
\end{counterexample}
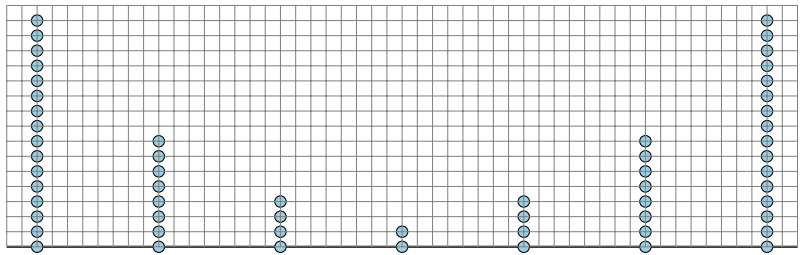
\begin{figure}[h!]
\begin{tikzpicture}[scale=0.80]%[every shadow/.style={fill=black!30,shadow xshift=0.2ex,shadow yshift=-0.2ex}]
\tikzstyle{redcirc}=[circle,
%circular drop shadow,
draw=black,fill=myred,thin,inner sep=0pt,minimum size=1.5mm]
\tikzstyle{bluecirc}=[circle,
%circular drop shadow,
draw=black,fill=myblue,thin,inner sep=0pt,minimum size=1.5mm]

\draw [thick] (-6.5,0) to (6.5,0);

\node (v1) at (0,0) [bluecirc] {};
%\node (v2) at (0.25,0) [bluecirc] {};
%\node (v3) at (0.5,0) [bluecirc] {};
%\node (v4) at (0.5,0.25) [bluecirc] {};
\node (v5) at (0,0.25) [bluecirc]{};
%\node (v6) at (0.5,-0.25) [bluecirc] {};
%\node (v7) at (0,-0.25) [bluecirc]{};

%%%%%%%%%%%%%%%%%%%%%%%%%%%%%%%%%%%%%%%%

\node (v10) at (2,0) [bluecirc] {};
\node (vm10) at (-2,0) [bluecirc] {};
\node (vm50) at (-2,0.25) [bluecirc]{};
\node (vm55) at (-2,0.5) [bluecirc]{};
\node (vm551) at (-2,0.75) [bluecirc]{};
\node (v551) at (2,0.75) [bluecirc]{};

\node (v55) at (2,0.5) [bluecirc]{};

%\node (v20) at (2.25,0) [bluecirc] {};
%\node (v25) at (2.5,0.5) [bluecirc] {};
%\node (v30) at (2.5,0) [bluecirc] {};
%\node (v40) at (2.5,0.25) [bluecirc] {};
\node (v50) at (2,0.25) [bluecirc]{};

%\node (v60) at (2.5,-0.25) [bluecirc] {};
%\node (v65) at (2.5,-0.5) [bluecirc] {};
%\node (v70) at (2,-0.25) [bluecirc]{};
%\node (v75) at (2,-0.5) [bluecirc]{};

%%%%%%%%%%%%%%%%%%%%%%%%%%%%%%%%%%%%%%%%

\node (v100) at (4,0) [bluecirc] {};
%\node (v200) at (4.25,0) [bluecirc] {};
%\node (v250) at (4.5,0.5) [bluecirc] {};
%\node (v300) at (4.5,0) [bluecirc] {};
%\node (v400) at (4.5,0.25) [bluecirc] {};
\node (v500) at (4,0.25) [bluecirc]{};
\node (v550) at (4,0.5) [bluecirc]{};

\node (vm100) at (-4,0) [bluecirc] {};
\node (vm500) at (-4,0.25) [bluecirc]{};
\node (vm550) at (-4,0.5) [bluecirc]{};

\node (v101) at (4,0.75) [bluecirc]{};
\node (vm101) at (-4,0.75) [bluecirc]{};
%\node (v102) at (4.5,0.75) [bluecirc]{};

\node (v103) at (4,1) [bluecirc] {};
\node (v104) at (4,1.25) [bluecirc] {};
\node (v105) at (4,1.5) [bluecirc] {};
\node (v106) at (4,1.75) [bluecirc] {};
%\node (v107) at (4,2) [bluecirc] {};
\node (vm103) at (-4,1) [bluecirc] {};
\node (vm104) at (-4,1.25) [bluecirc] {};
\node (vm105) at (-4,1.5) [bluecirc] {};
\node (vm106) at (-4,1.75) [bluecirc] {};
%\node (vm107) at (-4,2) [bluecirc] {};

%\node (v303) at (4.5,-0.25) [bluecirc] {};
%\node (v304) at (4.5,-0.5) [bluecirc] {};
%\node (v305) at (4.5,-0.75) [bluecirc] {};

%%%%%%%%%%%%%%%%%%%%%%%%%%%%%%%%%%%%%%%%
\node at (6,0) [bluecirc] {};
\node at (-6,0) [bluecirc] {};
\node at (6,0.25) [bluecirc] {};
\node at (-6,0.25) [bluecirc] {};
\node at (6,0.5) [bluecirc] {};
\node at (-6,0.5) [bluecirc] {};
\node at (6,0.75) [bluecirc] {};
\node at (-6,0.75) [bluecirc] {};
\node at (6,1) [bluecirc] {};
\node at (-6,1) [bluecirc] {};
\node at (6,1.25) [bluecirc] {};
\node at (-6,1.25) [bluecirc] {};
\node at (6,1.5) [bluecirc] {};
\node at (-6,1.5) [bluecirc] {};
\node at (6,1.75) [bluecirc] {};
\node at (-6,1.75) [bluecirc] {};
\node at (6,2) [bluecirc] {};
\node at (-6,2)[bluecirc] {};
\node at (6,2.25) [bluecirc] {};
\node at (-6,2.25) [bluecirc] {};
\node at (6,2.5) [bluecirc] {};
\node at (-6,2.5) [bluecirc] {};
\node at (6,2.75) [bluecirc] {};
\node at (-6,2.75) [bluecirc] {};
\node at (6,3) [bluecirc] {};
\node at (-6,3) [bluecirc] {};
\node at (6,3.25) [bluecirc] {};
\node at (-6,3.25) [bluecirc] {};
\node at (6,3.5) [bluecirc] {};
\node at (-6,3.5) [bluecirc] {};
\node at (6,3.75) [bluecirc] {};
\node at (-6,3.75) [bluecirc] {};
%\node at (6,4) [bluecirc] {};
%\node at (-6,4) [bluecirc] {};
%\node at (6,4.25) [bluecirc] {};
%\node at (-6,4.25) [bluecirc] {};
%\node at (6,4.5) [bluecirc] {};
%\node at (-6,4.5) [bluecirc] {};
%\node at (6,4.75) [bluecirc] {};
%\node at (-6,4.75) [bluecirc] {};

%\draw [cyan, xshift=4cm] plot [smooth, tension=2] coordinates { (0,0) (1,1) (2,-2) (3,0)};

\draw[step=0.25cm,gray,very thin] (-6.5,0) grid (6.5,4.00);
%\draw[step=0.25cm,gray,very thin] (-0.2,-0.2) grid (13.9,-1.49);

\end{tikzpicture}
\caption{Example of a big set with zero stationary harmonic measure at any point\label{fig:counter}}
\end{figure}
%\begin{remark}
%The monotonicity in Proposition 2 in \cite{Stationary_Harmonic_Measure} does not contradict with Counterexample \ref{counterexample_0}. The reason is that in the proof of Proposition 2 we need to interchange the order of a (finite) summation and a limit, which is not true for the infinite summation in Counterexample \ref{counterexample_0}.
%\end{remark}

In this paper, we will concentrate on characterizing the infinite subsets with zero/nonzero stationary harmonic measure. Our first results is a much stronger statement than Counterexample \ref{counterexample_0}: for any (infinite) $B\subset \upspace$, and any $x_1\in \ZZ$, define 
$$
h_{x_1}=\sup\{x_2\ge 0 , \ x_1\times[1, x_2]\subset B\}. 
$$

\begin{definition}
\label{definition: linear}
We say that $B$ has a {\bf horizontal super-linear growth} if there are constants $c\in (0,\infty)$, and $M<\infty$ such that 
$$
h_{x_1}\ge |cx_1|
$$
for all $|x_1|\ge M$.% and that 
%$$
%\left|B\setminus\{x\in\upspace, \ x_2\le C|x_1| \}\right|<\infty. 
%$$
\end{definition}

Then we have 
\begin{theorem}
\label{proposition_0}
For any $B$ which has a horizontal super-linear growth and any $x\in B$
$$
\sharm_B(x)=0. 
$$
\end{theorem}
With Theorem \ref{proposition_0}, Counterexample \ref{counterexample_0} is immediate. On the other hand, we prove that for $B$'s of which the spatial growth rate has (some) sub-linear upper bound, $\sharm_B(\cdot)$ cannot be 0 everywhere: 
\begin{theorem}
\label{proposition_not_0}
For $B\subset\upspace$ if there exists an $\alpha>1$ such that 
$$
\left|B\setminus\{x\in\upspace, \ x_2\le |x_1|^{1/\alpha} \}\right|<\infty,
$$
then there must be some $x\in B$ such that $\sharm_B(x)>0$. 
\end{theorem}

\begin{remark}
The conditions of Theorem \ref{proposition_0} are not essential as any linear stretching of Counterexample 1 (as can be seen in Figure \ref{fig:counter})
 will have uniformly zero stationary harmonic measure. There is also a gap between the conditions of Theorems \ref{proposition_0} and \ref{proposition_not_0}. It would be interesting to obtain sharp conditions for sets of zero stationary harmonic measure.
\end{remark}

\section{Proof of Theorem  \ref{proposition_0}}

For any $B$ with a horizontal super-linear growth and any $x=(x_1,x_2)\in B$.
Recall Definition \ref{definition: linear} and let 
$$
n_1=\max\{|x_1|,M, \lceil x_2/c\rceil \}
$$
and 
$$
D_1=[-n_1,n_1]\times[-\lceil cn_1\rceil, \lceil cn_1\rceil]. 
$$
Then $x\in D_1$, and by Definition \ref{definition: linear}, 
\beq
\label{theorem_1_1}
\hat W_c\setminus D_1\subset B\setminus D_1
\eeq
where 
$$
\hat W_c=\{x\in \upspace, \ x_2< c|x_1|\}. 
$$
Moreover, it is not hard to check that for any $N>\lceil cn_1\rceil$ and $y\in L_N\setminus B$, a simple random walk starting from $y$ hits $x$ before hitting any other point in $B$ only if it hits $l_{n_1}=[-n_1,n_1]\times \lceil cn_1\rceil=L_{\lceil cn_1\rceil}\setminus \hat W_c$ before hitting $\hat W_c$. I.e., 
\beq
\label{theorem_1_2}
P_y(\tau_{x}=\tau_B)\le P_y(\tau_{l_{n_1}}<\tau_{\hat W_c}). 
\eeq
Thus by \eqref{theorem_1_1} and \eqref{theorem_1_2}
$$
\begin{aligned}
\sharm_{B,N}(x)&=\sum_{y\in L_N\setminus B}P_y(\tau_{x}=\tau_B)\\
&\le \sum_{y\in L_N\setminus \hat W_c} P_y(\tau_{l_{n_1}}<\tau_{\hat W_c})\\
&=\sum_{w\in l_{n_1}} \sharm_{l_{n_1}\cup \hat W_c,N}(w).  
\end{aligned}
$$

Then by the proof of Proposition 1 in \cite{Stationary_Harmonic_Measure} (which is based on the time reversal argument used in \cite{harmonic_measure_1987}), for any $w\in l_{n_1}$
\beq
\label{theorem_1_3}
\begin{aligned}
&\sharm_{l_{n_1}\cup \hat W_c,N}(w)\\
=\sum_{y\in L_N\setminus \hat W_c} &P_{w}\Big(\tau_{L_N}<\tau_{l_{n_1}\cup\hat W_c}, S_{\tau_{L_N}}=y\Big)E_y\left[\text{number of visits to $L_N$ in }[0,\tau_{l_{n_1}\cup\hat W_c})\right]\\
\le \sum_{y\in L_N\setminus \hat W_c} & P_{w}\Big(\tau_{L_N}<\tau_{\hat W_c}, S_{\tau_{L_N}}=y\Big) E_y\left[\text{number of visits to $L_N$ in }[0,\tau_{L_0})\right]\\
=4N\cdot P_{w}&\Big(\tau_{L_N}<\tau_{\hat W_c}\Big). 
\end{aligned}
\eeq
Now let $N_1=\lceil cn_1\rceil$ and $N_2=2N_1$. For any $z\in l_{n_1}$ define the rectangular region 
$$
D_{1,z}=\big[z-4\lceil N_2/c\rceil, z+4\lceil N_2/c\rceil\big]\times [0,N_2].
$$
Moreover, we define the four sides on the boundary of $D_{1,z}$
$$
\begin{aligned}
&\partial^1D_{1,z}=\big[z-4\lceil N_2/c\rceil, z+4\lceil N_2/c\rceil\big]\times N_2,\\
&\partial^2D_{1,z}=\big(z+4\lceil N_2/c\rceil\big)\times [0,N_2],\\
&\partial^3D_{1,z}=\big[z-4\lceil N_2/c\rceil, z+4\lceil N_2/c\rceil\big]\times 0,\\
&\partial^4D_{1,z}=-\big(z+4\lceil N_2/c\rceil\big)\times [0,N_2]\\
\end{aligned}
.$$
Note that if a random walk starting at $z$ hits $\partial^2D_{1,z}\cup\partial^3D_{1,z}\cup \partial^4D_{1,z}$ before hitting $\partial^1D_{1,z}$, it must have already hit $\hat W_c$ before reaching $L_{N_2}$. Thus
$$
P_{z}\Big(\tau_{L_{N_2}}<\tau_{\hat W_c}\Big)\le P_z\Big(\tau_{\partial^1D_{1,z}}=\tau_{\partial D_{1,z}}\Big). 
$$
Then by translation invariance we have  
$$
P_z\Big(\tau_{\partial^1D_{1,z}}=\tau_{\partial D_{1,z}}\Big)=P_0\Big(\tau_{\partial^1D_{1,0}}=\tau_{\partial D_{1,0}}\Big).
$$
And by symmetry
\beq
\label{theorem_1_3}
P_{(0,N_1)}\Big(\tau_{\partial^1D_{1,0}}=\tau_{\partial D_{1,0}}\Big)=\frac{1}{2}-\frac{1}{2}\cdot P_0\Big(\tau_{\partial^2D_{1,0}}\wedge \tau_{\partial^4D_{1,0}}<\tau_{\partial^1D_{1,0}}\wedge \tau_{\partial^3D_{1,0}}\Big).
\eeq
Note that the last term in \eqref{theorem_1_3} is the probability a random walk first reaches the two vertical sides of $D_{1,z}$ before the horizontal sides. By invariance principle, there is a constant $c>0$ {\bf independent to} $N_1$ such that 
\beq
\label{theorem_1_4}
P_{(0,N_1)}\Big(\tau_{\partial^1D_{1,0}}=\tau_{\partial D_{1,0}}\Big)\le \frac{1-c}{2}. 
\eeq
In general, define $N_k=2^{k-1} N_1$ for all $k\ge 2$, and let 
$$
l_{N_k}=L_{N_k}\setminus \hat W_c,
$$
$$
D_{k,z}=\big[z-4\lceil N_{k+1}/c\rceil, z+4\lceil N_{k+1}/c\rceil\big]\times [0,N_{k+1}],
$$
with $\partial^1D_{k,z}$ - $\partial^4D_{k,z}$ as its four sides defined as before. Using exactly the same argument as for $k=1$, we have for any $z\in l_{N_k}$ (See Figure \ref{fig:notmono}), 
\beq
\label{theorem_1_5}
P_{z}\Big(\tau_{L_{N_{k+1}}}<\tau_{\hat W_c}\Big)\le P_{(0,N_k)}\Big(\tau_{\partial^1D_{k,0}}=\tau_{\partial D_{k,0}}\Big)\le \frac{1-c}{2}.
\eeq
\begin{figure}[h!]
\centering
\begin{tikzpicture}[scale=0.80]%[every shadow/.style={fill=black!30,shadow xshift=0.2ex,shadow yshift=-0.2ex}]
\tikzstyle{redcirc}=[circle,
%circular drop shadow,
draw=black,fill=myred,thin,inner sep=0pt,minimum size=2mm]
\tikzstyle{bluecirc}=[circle,
%circular drop shadow,
draw=black,fill=myblue,thin,inner sep=0pt,minimum size=2mm]

\node (v1) at (2.5,1.6) {$N2^k$};
\node (v2) at (4.5,3.6) {$N2^{k+1}$};
\node (v3) at (1.5,2) [redcirc] {};
%\node (v3) at (2,0) [bluecirc] {\small{$7$}};
%\node (v4) at (3,0) [bluecirc] {\small{$6$}};
%\node (v5) at (4,0) [bluecirc]{\small{$5$}};
%\node (v6) at (5,0) [bluecirc] {\small{$2$}};
%\node (v7) at (5,1) [bluecirc] {\small{$3$}};
%\node (v8) at (6,1) [bluecirc] {\small{$4$}};
%\node (v9) at (7,1) [bluecirc] {\small{$1$}};
%\node (v10) at (6,2) [redcirc] {\small{$1'$}};

\draw [thick] (-4.5,0) to (7.5,0);
\draw (0,0) to (-5,5);
\draw (0,0) to (5,5);
\draw [thick,blue] (-2,2) to (2,2);
\draw [thick,blue] (-4,4) to (4,4);
\draw [dashed,blue] (v3) to (7.5,2);
\draw [dashed,blue] (v3) to (-4.5,2);
\draw [dashed,blue] (1.5,4) to (7.5,4);
\draw [dashed,blue] (1.5,4) to (-4.5,4);
\draw [dashed,blue] (7.5,0) to (7.5,4);
\draw [dashed,blue] (-4.5,0) to (-4.5,4);

\draw [myred,thick] plot [smooth, tension=2.5] coordinates { (v3) (1.5,2.5)(1.3,2.6)(1.2,2.7)(1.7,2.9) (1.7,2.5)(2.3,2.7)(2.6,3)(2.7,3.2) (2.3,3.5)(3,3.7) (3.2,4)};

%\draw[step=1cm,gray,very thin] (-0.9,-0.9) grid (7.9,2.9);
\end{tikzpicture}
\caption{Escaping probability for each step}
\label{fig:notmono}
\end{figure}
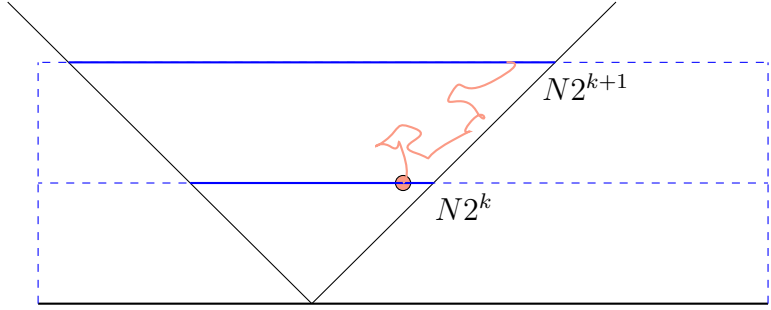 

Noting that the upper bound in \eqref{theorem_1_5} is uniform for all $z\in L_{N_k}\setminus \hat W_c$, by strong Markov property we have for any $w\in l_{n_1}$
\beq
\label{theorem_1_6}
P_{w}\Big(\tau_{L_{N_{k}}}<\tau_{\hat W_c}\Big)\le \left(  \frac{1-c}{2}\right)^{k-1}=2^{-(k-1)(1+\gamma)}
\eeq
where $\gamma=-\log_2(1-c)>0$. Recalling that $N_k=2^{k-1} N_1$, \eqref{theorem_1_3} and \eqref{theorem_1_6} give us 
\beq
\label{theorem_1_7}
\sharm_{l_{n_1}\cup \hat W_c,N_k}(w)\le 2^{k+1-(1+\gamma)(k-1)} N_1=2^{-\gamma(k-1)+2}N_1\to 0
\eeq
as $k\to\infty$. Thus the proof of Theorem \ref{proposition_0} is complete. \qed

\section{Proof of Theorem \ref{proposition_not_0}}

Recall that for $B$ in Theorem \ref{proposition_not_0} there exist an  $\alpha>1$ such that 
$$|\hat B|=\left|B\setminus\{x\in\upspace, \ x_2\le |x_1|^{1/\alpha} \}\right|<\infty.$$
Let $\QQ$ be all rational numbers. For technical reasons, consider 
$$
\QQ_1=\left\{\frac{\log(n_1)}{\log(n_2)}, \ n_1,n_2\in \ZZ, \ n_1,n_2\ge 2\right\}
$$
and 
$$
\QQ_2=\left\{\frac{b-dx}{cx-a}\in \RR,x\in \QQ_1, \ a,b,c,d\in \ZZ \right\}\supset \QQ_1
$$
which are both countable by definition. Now one can without loss of generality assume that $\alpha\notin \QQ\cup \QQ_2$. Thus for any $a,b,c,d\in \ZZ$, and $\alpha'=(a\alpha+b)/(c\alpha+d)$, as long as $a^2+b^2, \ c^2+d^2>0$, we always have $\alpha'\notin \QQ\cup\QQ_1$, which implies that for all integers $m\not=n\ge 1$, $m^{\alpha'}\not=n$.

Then we define $\bar B=B\setminus \hat B$, and
$$
h_0=\left\{
\begin{aligned}
&\max_{x\in \hat B} \{x_2\}, \ \ \hspace{2.57 in} \text{if } \hat B\not=\emptyset\\
&\min\left\{h\in \ZZ: \ B\cap \left[-\lceil h^\alpha\rceil, \lceil h^\alpha\rceil \right]\times [0,h]\not=\emptyset\right\}, \ \ \text{if } \hat B=\emptyset.
\end{aligned}
\right.
$$
Note that if $\hat B\not=\emptyset$, for each $N>h_0$ and $|i|\le \lfloor N^\alpha \rfloor$, note that $|i|\le \lfloor N^\alpha\rfloor <N^\alpha$ by the definition of $\alpha$, and that $N>h_0=\max_{x\in \hat W} \{x_2\}$. We have
$$
\{(i,N), \ |i|\le \lfloor N^\alpha\rfloor\}\cap B=\emptyset. 
$$
And if $\hat B=\emptyset$, $h_0$ is always finite since $B$ is nonempty. Then define 
$$
D_0=\left[-\lceil h_0^\alpha\rceil, \lceil h_0^\alpha\rceil \right]\times [0,h_0],
$$
and 
$$
B_0=B\cap D_0.
$$
We have $1\le |B_0|<\infty$ and $\hat B\subset B_0$. Now to prove Theorem \ref{proposition_not_0}, we only need to show that there is a constant $c>0$ such that for all sufficiently large $N$, 
\beq
\label{lower_bound_}
\sharm_{B,N}(B_0)\ge c.
\eeq
Let $l_{n}=\left[-\lfloor n^\alpha\rfloor, \lfloor n^\alpha\rfloor \right]\times n$. We first prove that 
\begin{lemma}
\label{lemma: harmonic}
There is a constant $c>0$ such that for any sufficiently large $N$
$$
\sharm_{B\cup l_{h_0},N}(l_{h_0})\ge c. 
$$
\end{lemma}

\begin{proof}
Recall that 
$$
D_0=\left[-\lceil h_0^\alpha\rceil, \lceil h_0^\alpha\rceil \right]\times [0,h_0]. 
$$
Now define
$$
U_{h_0}=\{x\in \upspace, \ |x_1|\ge \lceil h_0^\alpha\rceil, x_2\le |x_1|^{1/\alpha}\}
$$
and
$$
\hat W=D_0\cup U_{h_0}, \ W=\upspace\setminus(D_0\cup U_{h_0}). 
$$
Note that  $(\pm \lceil h_0^\alpha\rceil, h_0)\in U_{h_0}$, which implies that $\hat W$ is connected, and that 
$$
W=\{x\in \upspace, \ |x_2|>h_0, \ x_2\ge |x_1|^{1/\alpha}\}.
$$
An illustration of the subsets above can be seen in Figure \ref{fig:subset}.
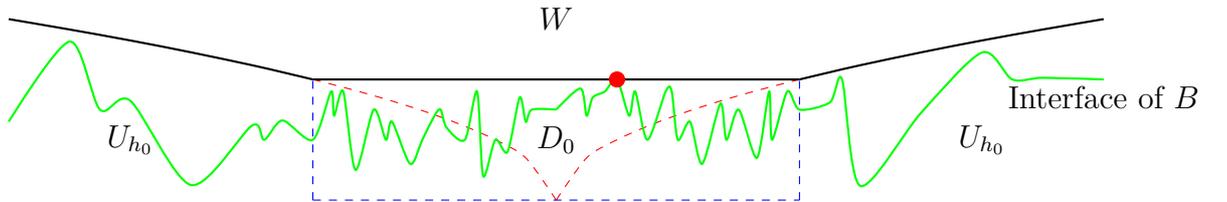
\begin{figure}[h!]
	\begin{tikzpicture}[scale=0.80]%[every shadow/.style={fill=black!30,shadow xshift=0.2ex,shadow yshift=-0.2ex}]
\tikzstyle{redcirc}=[circle,
%circular drop shadow,
draw=red,fill=red,thin,inner sep=0pt,minimum size=2mm]
\tikzstyle{bluecirc}=[circle,
%circular drop shadow,
draw=black,fill=myblue,thin,inner sep=0pt,minimum size=2mm]

\draw [thick] plot [smooth] coordinates {(4,2) (4.5,2.12)(5,2.236) (6,2.449) (7,2.646) (8,2.828) (9,3) };
\draw [thick] plot [smooth] coordinates {(-4,2) (-4.5,2.12)(-5,2.236) (-6,2.449) (-7,2.646) (-8,2.828) (-9,3)};

\draw [dashed,red] plot [smooth] coordinates { (0,0) (0.5,0.7)(1,1)(1.5,1.224)(2,1.414) (2.5,1.581)(3,1.7321)(3.5,1.87)(4,2)};
\draw [dashed,red] plot [smooth] coordinates { (0,0) (-0.5,0.7)(-1,1)(-1.5,1.224)(-2,1.414) (-2.5,1.581)(-3,1.7321)(-3.5,1.87)(-4,2)};

\draw [dashed,blue] (-4,0) to (4,0);

\draw [dashed,blue] (-4,0) to (-4,2);

\draw [dashed,blue] (4,0) to (4,2);

\draw [thick] (-4,2) to (4,2);

\filldraw[fill=red](0,0) (1,1) (2,0);

\node (v1) at (0,3) {$W$};

\node (v2) at (-7,1) {$U_{h_0}$};

\node (v3) at (7,1) {$U_{h_0}$};

\node (v4) at (0,1) {$D_0$};

\node (v4) at (9,1.7) {Interface of $B$};

\draw [thick, green] plot [smooth] coordinates {(0,1.5) (0.4,1.85) (0.5,1.4) (0.6,1.7) (0.8, 1.8) (1,2) (1.2,1.4) (1.3,1.8) (1.5,1) (1.8,1.8) (1.9,1.8) (2,1) (2.2, 1.3) (2.4,0.6) (2.7,1.6) (2.8,1) (3,1.5) (3.3,0.7) (3.5,1.8) (3.55,1) (3.8,1.8) (4,1.5) (4.5,2.12-0.5) (4.7,2) (5,2.236-2) (6,2.449-1) (7,2.646-0.2) (7.5,2) (8,2.828-0.8) (9,2)};

\node (v5) at (1,2) [redcirc]{};

\draw [thick, green] plot [smooth] coordinates {(0,1.5) (-0.4,1.5) (-0.5,1.3) (-0.6,1.7) (-0.8, 0.8) (-1,1) (-1.2,0.4) (-1.3,1.8) (-1.5,1) (-1.8,1.2) (-1.9,1.5) (-2,1.4) (-2.2, 1) (-2.4,0.6) (-2.7,1.3) (-2.8,1) (-3,1.5) (-3.3,0.5) (-3.5,1.8) (-3.65,1.4) (-3.7,1.8) (-4,1) (-4.5,2.12-0.8) (-4.8,1) (-5,2.236-1) (-6,2.449-2.2) (-7,2.646-1) (-7.5,1.5) (-8,2.828-0.2) (-9,1.3)};

\end{tikzpicture}
	\caption{Illustration of the subsets $B$, $D_0$, $U_{h_0}$ and $W$\label{fig:subset}}
\end{figure}

At the same time, recall that 
$$
\bar B\subset \{x\in \upspace, x_2\le |x_1|^{1/\alpha}\}.
$$
Combining this with the fact that $\hat B\subset D_0$ shown above, one has $B\cup l_{h_0}\subset \hat W$, which implies that 
$$
\sharm_{B\cup l_{h_0},N}(l_{h_0})\ge \sharm_{\hat W,N}(l_{h_0})\ge  \sharm_{\hat W,N}(\xi_0),
$$
where $\xi_0=(0,h_0)$. So for Lemma \ref{lemma: harmonic}, it suffices to prove that, there is a constant $c>0$, independent to $N$ such that for all sufficiently large $N$,
\beq
\label{harmonic_reformat}
\sharm_{\hat W,N}(\xi_0)\ge c.
\eeq 
The rest of this proof will concentrate on showing \eqref{harmonic_reformat}. First, by the existence of the limit \cite[Proposition 1]{Stationary_Harmonic_Measure}, it suffices to show the inequality for a subsequence $N_k\uparrow \infty$, say $N_k=2^{k}$, i.e., for all sufficiently large $k$, 
\beq
\label{harmonic_reformat_2}
\sharm_{\hat W,2^k}(\xi_0)\ge c.
\eeq
Moreover, for each sufficiently large $k$, noting that all $x\in L_{2^k}$ such that $|x_1|>2^{\alpha k}$ is in $\hat W$, 
\beq
\label{lemma_1_1}
\begin{aligned}
\sharm_{\hat W,2^k}(\xi_0)&=\hspace{-0.34 in}\sum_{z=(i,2^k), \ |i|\le \lfloor 2^{\alpha k}\rfloor}P_z(\tau_{\xi_0}\le\tau_{\hat W})\\
&=\sum_{z\in l_{2^k}} P_{\xi_0}\Big(\tau_{L_{2^k}}<\tau_{\hat W}, S_{\tau_{L_{2^k}}}=z\Big)E_z[\text{\# of visits to $l_{2^k}$ in }[0,\tau_{\hat W})].
\end{aligned}
\eeq
Intuitively speaking, we find a middle section $s \subset l_{2^k}$, whose formal definition will be presented in the later arguments (see \eqref{s_i} for details), and construct an event $A$ such that 
$$
A\subset \left\{ \tau_{L_{2^k}}<\tau_{\hat W}, \ S_{\tau_{L_{2^k}}}\in s\right\}, P(A)\ge c 2^{-k}. 
$$
At the same time, we show that for each $z\in s$, 
$$
E_z[\text{\# of visits to $l_{2^k}$ in }[0,\tau_{\hat W})]\ge c 2^k.
$$
In order to achieve this, one can put the trajectory between each $L_{2^i}$ and $L_{2^{i+1}}$ within appropriately chosen linear wedges such that 
\begin{enumerate} [(i)] 
\item The slope of these wedges flatten out to 0, which make the success probability from $L_{2^i}$ to $L_{2^{i+1}}$ close to $1/2$. 
\item The flattening out rate is slower than $2^{i(\alpha-1)}$ so all the linear wedges are still confined within the middle section of $l_{2^{i+1}}$. 
\end{enumerate}
To carry out this outline, we first define several parameters needed later in the construction. Recalling that $\alpha>1$, let $\beta=4/(\alpha+3)\in (0,1)$, and $\gamma=2(\alpha-1)/(\alpha+3)\in (0,2)$. Note that 
$$
\beta\alpha-1=\frac{4\alpha}{\alpha+3}-1=\frac{3\alpha-3}{\alpha+3}>\gamma>0,
$$
while at the same time
$$
\beta+\gamma=\frac{2\alpha+2}{\alpha+3}=: \alpha_1>1. 
$$
Let 
$$
k_0=\min\{k: \ 2^{\beta k}>2h_0\}\vee2\left\lceil\frac{\alpha+3}{\alpha-1} \right\rceil\vee \left\lceil \frac{3}{\alpha_1-1}\right\rceil.  
$$
For now, the restrictions above may look mysterious, but we will show the meaning of each of them along our proof.  

The following lemma from calculus is used repeatedly in our arguments:
\begin{lemma}
\label{lemma_calculus}
For all $x,y\ge 0$ and $\alpha>1$, 
\beq
\label{calculus}
(x+y)^{\alpha}\ge x^{\alpha}+\alpha  x^{\alpha-1} y. 
\eeq
\end{lemma}

\begin{proof}
Note that \eqref{calculus} is clearly true when $xy=0$. Assuming $x,y>0$, by mid value theorem and the fact that $x^{\alpha-1}$ is increasing, we have 
 $$
\begin{aligned}
(x+y)^{\alpha}&=x^\alpha \left(1+\frac{y}{x}\right)^\alpha\ge x^\alpha \left(1+\alpha \frac{y}{x} \right)=x^{\alpha}+\alpha  x^{\alpha-1} y. 
\end{aligned}
$$
\end{proof}

{\bf Lower bounds for the escaping probabilities:} Now back to the proof of Lemma \ref{lemma: harmonic}. Recalling that the definition of $\hat W$ and $k_0$ is independent to the choice of $k$ in \eqref{harmonic_reformat_2}, consider the event 
$$
A_0=\left\{\tau_{L_{2^{k_0}}}<\tau_{\hat W_0}, \  S_{\tau_{L_{2^{k_0}}}}=(0,2^{k_0})\right\}.
$$
One has that the probability of $A_0$ is also a positive number independent to $k$. I.e., 
$$
P_{\xi_0}(A_0)=c>0,
$$
Now for $X_1=(0,2^{k_0})$, consider a new wedge 
$$
W_0=\left\{x=(x_1,x_2)\in \upspace, \ x_2-\lceil 2^{\beta k_0}\rceil\ge |x_1|\cdot2^{-\gamma k_0} \right\}.
$$
Note that $x_2\ge \lceil 2^{\beta k_0}\rceil\ge 2h_0$, $x\notin D_0$. At the same time, by Lemma \ref{lemma_calculus}, 
$$
\begin{aligned}
x_2^\alpha&\ge \left(\lceil 2^{\beta k_0}\rceil+ |x_1|\cdot2^{-\gamma k_0} \right)^\alpha\\
&\ge \lceil 2^{\beta k_0}\rceil^\alpha+ \alpha  \lceil 2^{\beta k_0}\rceil^{(\alpha-1)} |x_1| 2^{-\gamma k_0}\\
&\ge \lceil 2^{\beta k_0}\rceil^\alpha+ \alpha 2^{[\beta(\alpha-1)-\gamma]k_0}|x_1|. 
\end{aligned}
$$
Note that $\beta(\alpha-1)-\gamma\ge \beta\alpha-1-\gamma>0$. We have that $x_2^\alpha\ge \lceil 2^{\beta k_0}\rceil^\alpha+|x_1|>|x_1|$, which implies that $W_0\subset W$. 

For $k_1=k_0+1$, and probability 
$$
p_0=P_{\xi_0}\left(\tau_{L_{2^{k_1}}}<\tau_{W_0^c}\right),
$$
one can see that 
$$
p_0\le P_{\xi_0}(\tau_{L_{2^{k_1}}}<\tau_{\hat W}). 
$$
Moreover, we have that 
$$
W_0\cap L_{2^{k_1}}\subset s_1:=\{(x_1,2^{k_1}), \  |x_1|\le 2^{k_1(1+\gamma)}\}. 
$$
Then for each $y=(y_1,y_2)\in s_1$, we can always define a wedge
$$
W_{1,y}=\left\{(x_1,x_2)\in \upspace, \ x_2-\lceil2^{\beta k_1}\rceil\ge |x_1-y_1|\cdot2^{-\gamma k_1} \right\}.
$$
And again by Lemma \ref{lemma_calculus} and the fact that $\beta\alpha-1>\gamma$, for any $y\in s_1$ and any $x\in W_{1,y}$, if $x_1=y_1$, then 
\beq
\label{x_1=y_1}
x_2^\alpha\ge 2^{\beta\alpha k_1}> 2^{k_1(1+\gamma)}\ge |y_1|=|x_1|. 
\eeq
Otherwise, note that for any $a,b\ge 1$, 
\beq
\label{pre_calculus}
ab-(a+b-1)=(a-1)(b-1)\ge 0. 
\eeq
Thus
\beq
\label{W_{1,y}_1}
\begin{aligned}
x_2^\alpha &\ge  \left(\lceil2^{\beta k_1}\rceil+ |x_1-y_1|\cdot2^{-\gamma k_1} \right)^\alpha\\
&\ge \lceil 2^{\beta k_1}\rceil^\alpha+ \alpha  \lceil 2^{\beta k_1}\rceil^{(\alpha-1)} |x_1-y_1| 2^{-\gamma k_1}\\
&\ge \lceil 2^{\beta k_1}\rceil^\alpha+\alpha2^{[\beta(\alpha-1)-\gamma] k_1}|x_1-y_1| \\
&\ge  \lceil 2^{\beta k_1}\rceil^\alpha+\alpha2^{[\beta(\alpha-1)-\gamma] k_1}+|x_1-y_1| -1\\
&\ge  \lceil 2^{\beta k_1}\rceil^\alpha+\alpha2^{[\beta(\alpha-1)-\gamma] k_1}+|x_1|-|y_1| -1.
\end{aligned}
\eeq
It is known in \eqref{x_1=y_1} that $2^{\beta\alpha k_1}\ge 2^{k_1(1+\gamma)}\ge |y_1|$. At the same time, since $\beta(\alpha-1)-\gamma>0$, $\alpha2^{[\beta(\alpha-1)-\gamma] k_1}>1$. Thus 
\beq
\label{W_{1,y}_2}
x_2^\alpha\ge |x_1|+\left(\lceil 2^{\beta k_1}\rceil^\alpha-|y_1| \right)+\left(\alpha2^{[\beta(\alpha-1)-\gamma] k_1}-1\right)>|x_1|
\eeq
which implies that $W_{1,y}\subset W$ for all $y\in s_1$. Then for $k_2=k_1+1$ define the probability 
$$
p_{1,y}=P_{y}\left(\tau_{L_{2^{k_2}}}\le \tau_{W_{1,y}^c}\right). 
$$
By translation invariance, we have $p_{1,y}=p_1$ for all such $y$'s. In general, for all $i\ge 1$ let $k_i=k_0+i$. And for all 
\beq
\label{s_i}
y\in s_i=\left\{(y_1,2^{k_i}), |y_1|\le \sum_{j=1}^i 2^{(1+\gamma)k_j}\right\}
\eeq
we define wedge
$$
W_{i,y}=\left\{(x_1,x_2)\in \upspace, \ x_2-\lceil2^{\beta k_i}\rceil\ge |x_1-y_1|\cdot2^{-\gamma k_i} \right\}.
$$
For $x=(x_1,x_2)\in W_{i,y}$, one first assume $x_1\not=y_1$. Then 
\beq
\label{W_{i,y}_1}
\begin{aligned}
x_2^\alpha &\ge  \left(\lceil2^{\beta k_i}\rceil+ |x_1-y_1|\cdot2^{-\gamma k_i} \right)^\alpha\\
&\ge \lceil 2^{\beta k_i}\rceil^\alpha+ \alpha  \lceil 2^{\beta k_i}\rceil^{(\alpha-1)} |x_1-y_1| 2^{-\gamma k_1}\\
&\ge \lceil 2^{\beta k_i}\rceil^\alpha+\alpha2^{[\beta(\alpha-1)-\gamma] k_i}|x_1-y_1| \\
&\ge  \lceil 2^{\beta k_i}\rceil^\alpha+\alpha2^{[\beta(\alpha-1)-\gamma] k_i}+|x_1-y_1| -1\\
&\ge  \lceil 2^{\beta k_i}\rceil^\alpha+\alpha2^{[\beta(\alpha-1)-\gamma] k_i}+|x_1|-|y_1| -1.
\end{aligned}
\eeq
Similar to when $i=1$, we have $\alpha2^{[\beta(\alpha-1)-\gamma] k_i}>1$ while at the same time, 
$$
\begin{aligned}
|y_1|\le \sum_{j=1}^i 2^{(1+\gamma)k_j}\le \sum_{j=0}^{k_i} 2^{(1+\gamma)j}\le \frac{ 2^{(1+\gamma)(k_i+1)}-1}{2^{(1+\gamma)}-1}. 
\end{aligned}
$$
Note that $\gamma>0$, which implies that $2^{(1+\gamma)}-1\ge 2^\gamma$. Thus
\beq
\label{bound_s}
\begin{aligned}
|y_1|\le \frac{ 2^{(1+\gamma)(k_i+1)}-1}{2^{(1+\gamma)}-1}\le \frac{2^{(1+\gamma)(k_i+1)}}{2^\gamma}=2^{(1+\gamma)k_i+1}. 
\end{aligned}
\eeq
Now recall that 
$$
\beta\alpha-1-\gamma=\frac{\alpha-1}{\alpha+3}>0, 
$$
while 
$$
k_i>k_0\ge \left\lceil\frac{\alpha+3}{\alpha-1} \right\rceil. 
$$
We have $(\beta\alpha)k_i\ge (1+\gamma)k_i+1$, which implies $\lceil 2^{\beta k_i}\rceil^\alpha> |y_1|$ by the definition of $\alpha$ , and that $x_2^\alpha> |x_1|$. And if $x_1=y_1$, one can also have
$$
x_2^\alpha\ge \lceil 2^{\beta k_i}\rceil^\alpha> |y_1|=|x_1|. 
$$
Thus we have $W_{i,y}\subset W$. Also for each $y\in s_i$, and $z\in W_{i,y}\cap L_{2^{k_{i+1}}}$,
$$
|z_1|\le |y_1|+2^{k_{i+1}+\gamma k_i}\le \sum_{j=1}^{i+1} 2^{(1+\gamma)k_j}
$$
which implies that 
$$
\left(\bigcup_{y\in s_i}W_{i,y}\right)\cap L_{2^{k_{i+1}}}\subset s_{i+1}= \left\{(y_1,2^{k_{i+1}}), |y_1|\le \sum_{j=1}^{i+1} 2^{(1+\gamma)k_j}\right\}. 
$$
\begin{figure}[h!]
\centering
\begin{tikzpicture}[scale=0.80]%[every shadow/.style={fill=black!30,shadow xshift=0.2ex,shadow yshift=-0.2ex}]
\tikzstyle{redcirc}=[circle,
%circular drop shadow,
draw=black,fill=myred,thin,inner sep=0pt,minimum size=2mm]
\tikzstyle{bluecirc}=[circle,
%circular drop shadow,
draw=black,fill=myblue,thin,inner sep=0pt,minimum size=2mm]

\node (v1) at (7.5,1.6) {$L_{2^{k_j}}$};
\node (v2) at (7.5,3.6) {$L_{2^{k_{j+1}}}$};
\node (v5) at (7.5,-0.5) {$ L_{\lceil 2^{2k_j/3}\rceil}$};
\node (v3) at (0,2) [redcirc] {};
\node (v4) at (2.5,1.4) {$s_j$};
%\node (v5) at (4,0) [bluecirc]{\small{$5$}};
%\node (v6) at (5,0) [bluecirc] {\small{$2$}};
%\node (v7) at (5,1) [bluecirc] {\small{$3$}};
%\node (v8) at (6,1) [bluecirc] {\small{$4$}};
%\node (v9) at (7,1) [bluecirc] {\small{$1$}};
%\node (v10) at (6,2) [redcirc] {\small{$1'$}};

\draw [dashed] (-4.5,0) to (7.5,0);
\draw (0,0) to (-5,5);
\draw (0,0) to (5,5);
\draw [thick,blue] (-1.5,2) to (2.5,2);
\draw [thick,blue] (2.5,2.2) to (2.5,1.8);
\draw [thick,blue] (-1.5,2.2) to (-1.5,1.8);
\draw [thick,blue] (-4,4) to (4,4);
\draw [dashed,blue] (v3) to (7.5,2);
\draw [dashed,blue] (v3) to (-4.5,2);
\draw [dashed,blue] (1.5,4) to (7.5,4);
\draw [dashed,blue] (1.5,4) to (-4.5,4);
%\draw [dashed,blue] (7.5,0) to (7.5,4);
%\draw [dashed,blue] (-4.5,0) to (-4.5,4);

\draw [myred,thick] plot [smooth, tension=2.5,xshift=-1.5cm] coordinates { (1.5,2) (1.5,2.5)(1.3,2.6)(1.2,2.7)(1.7,2.9) (1.7,2.5)(2.3,2.7)(2.6,3)(2.7,3.2) (2.3,3.5)(3,3.7) (3.2,4)};

\draw [thick] plot [smooth,xshift=0.5cm,yshift=-1cm] coordinates { (0,0) (0.5,0.7)(1,1)(1.5,1.224)(2,1.414) (2.5,1.581)(3,1.7321)(3.5,1.87)(4,2) (4.5,2.12)(5,2.236)};
\draw [thick] plot [smooth,xshift=0.5cm,yshift=-1cm] coordinates { (0,0) (-0.5,0.7)(-1,1)(-1.5,1.224)(-2,1.414) (-2.5,1.581)(-3,1.7321)(-3.5,1.87)(-4,2) (-4.5,2.12)(-5,2.236)};

%\draw[step=1cm,gray,very thin] (-0.9,-0.9) grid (7.9,2.9);
\end{tikzpicture}
\caption{Escaping probability from $L_{2^{k_j}}$ to $L_{2^{k_{j+1}}}$}
\label{fig:notmono2}
\end{figure}
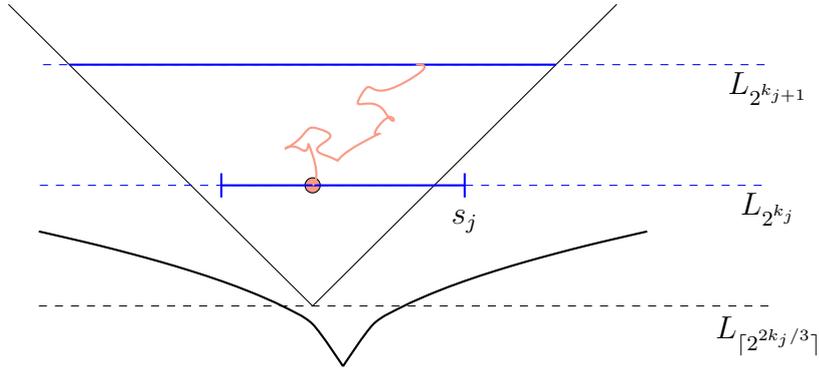 
And for all $y\in s_i$ by translation invariance, define (See Figure \ref{fig:notmono2})
$$
p_{i}=P_{y}(\tau_{L_{2^{k_{i+1}}}}\le \tau_{W_{i,y}^c}). 
$$
With the constructions above and strong Markov property, one can see that for each $i$
\beq
\label{lemma_1_2}
P_{\xi_0}\Big(\tau_{L_{2^{k_i}}}<\tau_{\hat W}\Big)\ge P_{\xi_0}\Big(\tau_{s_{i}}<\tau_{\hat W}\Big)\ge \prod_{j=0}^{i-1} p_j. 
\eeq
Now to find lower bounds for the success probability $p_i$, we need to the following simple lemma showing that it is highly unlikely for a simple random walk starting from the middle of a very wide but short rectangular box to exit from the vertical sides:  
\begin{lemma}
\label{lemma: rectangle}
For any integers $n,k\ge1$, let rectangle 
$$
R_{k,n}=[-nk,nk]\times [-k,k]
$$
with 
$$
l^v_{k,n}=\{-nk,nk\}\times [-k,k]
$$
as its two vertical sides and 
$$
l^h_{k,n}=[-nk,nk]\times \{-k,k\}
$$
as its two horizontal sides. Then there is a $\delta\in (0,1)$ such that for any $n,k\ge1$ and any integer $x\in \{0\}\times[-k,k]$, 
$$
P_x\left(\tau_{l^v_{k,n}}<\tau_{l^h_{k,n}}\right)\le (1-\delta)^n. 
$$
\end{lemma}

\begin{remark}
Similar bounds have been used in a later paper \cite{procaccia2018stationary}. 
\end{remark}

\begin{proof}
First we prove the inequality for $n=1$. For any $x\in \{0\}\times[-k,k]$, define
$$
R^{x}_{k,1}=[-k,k]\times [x_2-2k,x_2+2k],
$$
$$
l^{v,x}_{k,1}=\{-k,k\}\times [x_2-2k,x_2+2k],
$$
and
$$
l^{h,x}_{k,1}=[-k,k]\times \{x_2-2k,x_2+2k\}.
$$
Note that if a random walk starting at $x$ want to hit $l^{h,x}_{k,1}$ before hitting $l^{v,x}_{k,1}$, it has to hit $l^h_{k,1}$ before $l^v_{k,1}$. Thus
$$
P_x\left(\tau_{l^h_{k,1}}<\tau_{l^v_{k,1}}\right)\ge P_x\left(\tau_{l^{h,x}_{k,1}}<\tau_{l^{v,x}_{k,1}}\right).
$$
Then by translation and scaling invariance, there is a constant $\delta>0$ such that 
$$
P_x\left(\tau_{l^{h,x}_{k,1}}<\tau_{l^{v,x}_{k,1}}\right)=P_0\left(\tau_{l^{h,0}_{k,1}}<\tau_{l^{v,0}_{k,1}}\right)\ge \delta
$$  
for all $k$ and all $x\in \{0\}\times[-k,k]$. Now suppose we already have 
$$
P_x\left(\tau_{l^v_{k,n}}<\tau_{l^h_{k,n}}\right)\le (1-\delta)^n. 
$$
for $n$. For $n+1$, by strong Markov property we have 
$$
\begin{aligned}
P_x\left(\tau_{l^v_{k,n+1}}<\tau_{l^h_{k,n+1}}\right)&\le P_x\left(\tau_{l^v_{k,n}}<\tau_{l^h_{k,n}}\right)\sup_{x\in \{0\}\times[-k,k]} P_x\left(\tau_{l^v_{k,1}}<\tau_{l^h_{k,1}}\right)\\
&\le (1-\delta)^{n+1}.
\end{aligned}
$$ 
Thus the proof of Lemma \ref{lemma: rectangle} is complete. 
\end{proof}
With Lemma \ref{lemma: rectangle} we can bound from below the probabilities $p_i$. Recalling that by translation invariance, for each $i$, $y_{i,0}=(0,2^{k_i})$ and 
$$
W_{i,0}=\left\{(x_1,x_2)\in \upspace, \ x_2-\lceil2^{\beta k_i}\rceil\ge |x_1|\cdot2^{-\gamma k_i} \right\},
$$
we have
$$
p_i=P_{y_{i,0}}\left(\tau_{L_{2^{k_{i+1}}}}\le \tau_{W_{i,0}^c}\right).
$$
Then consider the rectangle 
$$
R_i=\Big[-\lfloor 2^{\alpha_1 k_i}\rfloor, \lfloor 2^{\alpha_1 k_i}\rfloor \Big]\times \Big[2\lceil 2^{\beta k_i}\rceil, 2^{k_i+1}\Big].
$$
Recalling that $k_0> 2(\alpha+3)/(\alpha-1)=2/(1-\beta)$, we have 
$$
2\lceil 2^{\beta k_i}\rceil< 2^{\beta k_i+2}\le 2^{k_i}
$$
and $R_i\not=\emptyset$. We claim that $R_i\subset W_{i,0}$. To show this, it suffices to check that the two corners at the bottom are within $W_{i,0}$. I.e., 
$$
\left( \lfloor 2^{\alpha_1 k_i}\rfloor,  2\lceil 2^{\beta k_i}\rceil \right)\in  W_{i,0}.
$$
To see this, note that by the definition of $\alpha$, 
$$
2^{\gamma k_i}\Big( 2\lceil 2^{\beta k_i}\rceil - \lceil 2^{\beta k_i}\rceil \Big)> 2^{(\beta+\gamma) k_i}
$$
and that $\alpha_1=\beta+\gamma$. Now let 
$$
\begin{aligned}
&\text{top}_i=\Big[-\lfloor 2^{\alpha_1 k_i}\rfloor, \lfloor 2^{\alpha_1 k_i}\rfloor \Big]\times 2^{k_i+1}\\
&\text{bottom}_i=\Big[-\lfloor 2^{\alpha_1 k_i}\rfloor, \lfloor 2^{\alpha_1 k_i}\rfloor \Big]\times 2\lceil 2^{\beta k_i}\rceil\\
&\text{left}_i=-\lfloor 2^{\alpha_1 k_i}\rfloor \times  \Big[2\lceil 2^{\beta k_i}\rceil, 2^{k_i+1}\Big]\\
&\text{right}_i=\lfloor 2^{\alpha_1 k_i}\rfloor \times  \Big[2\lceil 2^{\beta k_i}\rceil, 2^{k_i+1}\Big]. 
\end{aligned}
$$
Note that
$$
\frac{ \lfloor 2^{\alpha_1 k_i}\rfloor}{2^{k_i+1}}>  2^{(\alpha_1-1) k_i-2}\uparrow +\infty,
$$
and that $k_0\ge \lceil 3/(\alpha_1-1)\rceil$, which implies $2^{(\alpha_1-1) k_i-2}\ge 2$ for all $i$. Let 
$$
m=-2\lceil 2^{\beta k_i}\rceil+2^{k_i+1}, \ n=\left\lfloor \frac{\lfloor 2^{\alpha_1 k_i}\rfloor}{m} \right\rfloor.
$$
One can apply Lemma \ref{lemma: rectangle} to a translation of the box $R_{m,n}$ within $R_i$ and have 
\beq
\label{error_1}
P_{y_{i,0}}\left(\tau_{\text{left}_i\cup \text{right}_i}<\tau_{\text{top}_i\cup \text{bottom}_i}\right)\le (1-\delta)^{2^{(\alpha_1-1) k_i-2}}. 
\eeq
Moreover, we have 
\beq
\label{success_1}
P_{y_{i,0}}\left(\tau_{L_{2^{k_i+1}}}<\tau_{L_{2\lceil 2^{\beta k_i}\rceil}}\right)=\frac{2^{k_i}-2\lceil 2^{\beta k_i}\rceil}{2^{k_i+1}-2\lceil 2^{\beta k_i}\rceil}\ge \frac{1}{2}-2^{(\beta-1)k_i+1}. 
\eeq
Now note that 
$$
\left\{\tau_{\text{top}_i}=\tau_{\partial^{in} R_i}\right\}=\left\{\tau_{L_{2^{k_i+1}}}<\tau_{L_{2\lceil 2^{\beta k_i}\rceil}}\right\} \setminus \left\{\tau_{\text{left}_i\cup \text{right}_i}<\tau_{\text{top}_i\cup \text{bottom}_i}\right\}.
$$
We have by \eqref{error_1} and \eqref{success_1}, 
\beq
\label{error_2}
\begin{aligned}
p_i&\ge P_{y_{i,0}}\left(\tau_{\text{top}_i}=\tau_{\partial^{in} R_i}\right)\\
&\ge \frac{1}{2}-2^{(\beta-1)k_i+1}-(1-\delta)^{2^{(\alpha_1-1) k_i-2}}. 
\end{aligned}
\eeq
Now recalling \eqref{lemma_1_2}, we have 
$$
P_{\xi_0}\Big(\tau_{L_{2^{k_i}}}<\tau_{\hat W}\Big)\ge \prod_{j=0}^{i-1} p_j\ge  2^{-i}\prod_{j=0}^{i-1}\left[1-2^{(\beta-1)k_i+2}-2(1-\delta)^{2^{(\alpha_1-1) k_i-2}} \right]. 
$$
Noting that 
$$
\sum_{i=0}^\infty 2^{(\beta-1)k_i+2}+2(1-\delta)^{2^{(\alpha_1-1) k_i-2}}<\infty,
$$
there is a constant $c>0$ such that for all $i\ge 0$, 
\beq
\label{probability_escape}
P_{\xi_0}\Big(\tau_{s_{i}}<\tau_{\hat W}\Big)\ge c2^{-k_i}. 
\eeq

{\bf Lower bounds for the returning times:} Now Recall that 
$$
s_i=\left\{(y_1,2^{k_i}), |y_1|\le \sum_{j=1}^i 2^{(1+\gamma)k_j}\right\}\subset l_{2^{k_i}},
$$
and that since $\gamma=2(\alpha-1)/(\alpha+3)$, $1+\gamma=(3\alpha+1)/(\alpha+3)<\alpha$. For any $y=(y_1,,2^{k_i})\in s_i$, recalling the upper bound found in \eqref{bound_s}, we have 
\beq
\label{location_for_box}
|y_1|+ 2^{(1+\gamma)k_i+1}\le 2^{(1+\gamma)k_i+2}.
\eeq
Now note that 
$$
k_0\ge 2\left\lceil\frac{\alpha+3}{\alpha-1} \right\rceil> \frac{(\alpha+3)(\alpha+2)}{\alpha^2-1}=\frac{\alpha+2}{\alpha-\gamma-1},
$$
which implies that, $(k_i-1)\alpha>(1+\gamma)k_i+2$ for all $i\ge 0$ and that 
\beq
\label{space_for_box}
\left(2^{k_i-1} \right)^\alpha= 2^{(k_i-1)\alpha}> 2^{(1+\gamma)k_i+2}.
\eeq
Combining \eqref{location_for_box} and \eqref{space_for_box} gives us for all $y\in s_i$, 
$$
N_y^i=y+\left[-\lfloor2^{(1+\gamma)k_i+1}\rfloor,\lfloor2^{(1+\gamma)k_i+1}\rfloor\right]\times \left[-2^{k_i-1},2^{k_i-1}\right]\subset W. 
$$
And thus 
$$
E_y\left[\text{\# of visits to $L_{2^{k_i}}$ in }[0,\tau_{\hat W}]\right]\ge E_y\left[\text{\# of visits to $L_{2^{k_i}}$ in }[0,\tau_{\partial^{in} N^i_y}]\right]. 
$$
Now let $\Gamma_{i,1}=\tau_{L_{2^{k_i}}}$ and for each $j$
$$
\Gamma_{i,j}=\inf\left\{n>\Gamma_{i,j-1}, \ S_n\in L_{2^{k_i}} \right\}
$$
be the $j$th time a random walk returns to $L_{2^{k_i}}$. We have
$$
\begin{aligned}
E_y\left[\text{number of visits to $L_{2^{k_i}}$ in }[0,\tau_{\partial^{in} N^i_y}]\right]&=1+\sum_{j=1}^\infty P_y\left(\Gamma_{i,j}\le \tau_{\partial^{in} N^i_y}\right)\\
&\ge 1+\sum_{j=1}^{2^{k_i}} P_y\left(\Gamma_{i,j}\le \tau_{\partial^{in} N^i_y}\right).
\end{aligned}
$$
Again we define 
$$
\begin{aligned}
&\hat{\text{top}}_{y,i}=\left[y_1-\lfloor2^{(1+\gamma)k_i+1}\rfloor,y_1+\lfloor2^{(1+\gamma)k_i+1}\rfloor\right]\times (2^{k_i}+2^{k_i-1})\\
&\hat{\text{bottom}}_{y,i}=\left[y_1\lfloor2^{(1+\gamma)k_i+1}\rfloor,y_1+\lfloor2^{(1+\gamma)k_i+1}\rfloor\right]\times (2^{k_i}-2^{k_i-1})\\
&\hat{\text{left}}_{y,i}=\left(y_1-\lfloor2^{(1+\gamma)k_i+1}\rfloor\right)\times [2^{k_i}-2^{k_i-1},2^{k_i}+2^{k_i-1}] \\
&\hat{\text{right}}_{y,i}=\left(y_1+\lfloor2^{(1+\gamma)k_i+1}\rfloor\right)\times [2^{k_i}-2^{k_i-1},2^{k_i}+2^{k_i-1}] 
\end{aligned}
$$
as the four sides of $\partial^{in}N^i_y$. Note that for any $1\le j\le2^{k_i}$, 
$$
\begin{aligned}
P_y\left(\Gamma_{i,j}\le\tau_{\hat{\text{top}}_{y,i}}\wedge \tau_{\hat{\text{bottom}}_{y,i}} \right)&\ge P_y\left(\Gamma_{i,j}\le \tau_{L_{2^{k_i}+2^{k_i-1}}}\wedge\tau_{L_{2^{k_i}-2^{k_i-1}}} \right)\\
&=\left(1-2^{-k_i}\right)^j. 
\end{aligned}
$$
Moreover, 
$$
\begin{aligned}
P_y\left(\Gamma_{i,j}\le \tau_{\partial N^i_y}\right)&=P_y\left(\Gamma_{i,j}\le\tau_{\hat{\text{top}}_{y,i}}\wedge \tau_{\hat{\text{bottom}}_{y,i}} \right)\\
 &-P_y\left(\tau_{\hat{\text{left}}_{y,i}}\wedge \tau_{\hat{\text{right}}_{y,i}}< \Gamma_{i,j}\le\tau_{\hat{\text{top}}_{y,i}}\wedge \tau_{\hat{\text{bottom}}_{y,i}} \right)\\
 &\ge \left(1-2^{-k_i}\right)^j-P_y\left(\tau_{\hat{\text{left}}_{y,i}}\wedge \tau_{\hat{\text{right}}_{y,i}}<\tau_{\hat{\text{top}}_{y,i}}\wedge \tau_{\hat{\text{bottom}}_{y,i}} \right). 
 \end{aligned}
$$
Note that 
$$
\frac{\lfloor2^{(1+\gamma)k_i+1}\rfloor}{2^{k_i-1}}>2^{\gamma k_i}. 
$$
Again by Lemma \ref{lemma: rectangle}, we have 
$$
P_y\left(\tau_{\hat{\text{left}}_{y,i}}\wedge \tau_{\hat{\text{right}}_{y,i}}<\tau_{\hat{\text{top}}_{y,i}}\wedge \tau_{\hat{\text{bottom}}_{y,i}} \right)\le (1-\delta)^{2^{\gamma k_i}}.
$$
Thus,
\beq
\label{expected return}
\begin{aligned}
E_y&\left[\text{number of visits to $L_{2^{k_i}}$ in }[0,\tau_{\partial^{in} N^i_y}]\right]\\
&\ge \sum_{j=1}^{2^{k_i}} P_y\left(\Gamma_{i,j}\le \tau_{\partial^{in}  N^i_y}\right)\\
&\ge  \left(\sum_{j=1}^{2^{k_i}}\left(1-2^{-k_i}\right)^j\right)- 2^{k_i} (1-\delta)^{2^{\gamma k_i}}\\
&\ge 2^{k_i}\left(1-2^{-k_i}\right)\left[1-\left(1-2^{-k_i}\right)^{2^{k_i}}\right]-C\\
&\ge c 2^{k_i}
 \end{aligned}
\eeq
for some $c>0$ independent to $i$ and $y\in s_i$.  

Now combining \eqref{lemma_1_1}, \eqref{probability_escape} and \eqref{expected return}
\beq
\label{lemma_1_3}
\begin{aligned}
\sharm_{\hat W,2^{k_i}}(\xi_0)&=\sum_{z\in l_{2^{k_i}}} P_{\xi_0}\Big(\tau_{L_{2^{k_i}}}<\tau_{\hat W}, S_{\tau_{L_{2^{k_i}}}}=z\Big)E_z\left[\text{number of visits to $l_{2^{k_i}}$ in }[0,\tau_{\hat W})\right]\\
&\ge \sum_{y\in s_i} P_{\xi_0}\Big(\tau_{L_{2^{k_i}}}<\tau_{\hat W}, S_{\tau_{L_{2^{k_i}}}}=y\Big)E_y\left[\text{number of visits to $L_{2^{k_i}}$ in }[0,\tau_{\partial N^i_y}]\right]\\
&\ge P_{\xi_0}\Big(\tau_{L_{2^{k_i}}}<\tau_{\hat W}\Big) \inf_{y\in s_i}E_y\left[\text{number of visits to $L_{2^{k_i}}$ in }[0,\tau_{\partial N^i_y}]\right]\\
&\ge c. 
\end{aligned}
\eeq
And thus we have shown \eqref{harmonic_reformat_2} and the proof of Lemma \ref{lemma: harmonic} is complete. \end{proof}

Now back to finish the proof of Theorem \ref{proposition_not_0}, note that both $l_{h_0}$ and $B_0=B\cap D_0$ are finite and not depending on $N$. There is a $c>0$ such that for any $z\in l_{h_0}$, 
$$
P_z(\bar\tau_{B_0}=\bar\tau_{B})\ge c. 
$$
Thus by strong Markov property, 
$$
\begin{aligned}
\sharm_{B,2^{k_i}}(B_0)&=\sum_{z\in L_{2^{k_i}}\setminus B} P_z(\tau_{B_0}=\tau_B)\\
&\ge \sharm_{B\cup l_{h_0},2^{k_i}}(l_{h_0})\inf_{z\in l_{h_0}}P_z(\bar\tau_{B_0}=\bar\tau_{B})\\
&\ge c.
\end{aligned}
$$
Then taking $i\to\infty$, Proposition \ref{proposition_well_define} completes the proof of Theorem \ref{proposition_not_0}.  \qed

\section{Discussions}
%\note{need to simplify this part}
Now let's look back at the possible aggregation model. With Theorem \ref{proposition_not_0}, consider an interacting particle system first introduced in Proposition 3 of \cite{Stationary_Harmonic_Measure}: Let $\bar \xi_t$ defined on $\{0,1\}^\upspace$ with 1 standing for a site occupied while 0 for vacant, with transition rates as follows: 
\begin{enumerate}[(i)]
\item For each occupied site $x=(x_1,x_2)\in \upspace$, if $x_2>0$ it will try to give birth to each of its nearest neighbors at a Poisson rate of $\sqrt{x_2}$. If $x_2=0$, it will try to give birth to each of its nearest neighbors at a Poisson rate of $1$. 
\item When $x$ attempts to give birth to its nearest neighbors $y$ already occupied, the birth is suppressed.   
\end{enumerate}

In \cite{Stationary_Harmonic_Measure} we prove that $\bar \xi_t$ with transition rates above is a well defined infinite interacting particle system. And let 
$$
B_t=\{x\in \upspace, \ \bar\xi_t(x)=1\}. 
$$ 
Moreover, recalling that in the proof of Lemma 6.2 in \cite{Stationary_Harmonic_Measure}, for any $x\in \upspace$, $x\in L_0$ and $0\le t$, we define subset $I_{0,t}(x)$ as the collection of all possible offsprings of the particle at $x$ when at time $t$, and let 
$$
\mathcal{I}_{t,T}(x)=\sup_{y\in I_{t,T}(x)}|x-y|.
$$
When $B_0=L_0$, for any $x'\in B_t$, one can easily check by definition there must be an $x\in L_0$ such that $x'\in I_{0,t}(x)$, which implies that
$$
B_t=\bigcup_{x\in L_0}I_{0,t}(x). 
$$
Moreover, by (horizontal) translation invariance, we have $I_{0,t}(x)$ are identically distributed for all $x$. 

Then by Theorem 6 in \cite{Stationary_Harmonic_Measure}, we have for any $n\ge 1$
$$
E\left[\mathcal{I}_{0,t}(x)^{2n}\right]<\infty, 
$$
which implies that 
$$
\sum_{x\in L_0} P\left(\mathcal{I}_{0,t}(x)^{2n}\ge |x_1|\right)<\infty. 
$$
Then by Borel-Cantelli Lemma, with probability one for all $x\in L_0$ sufficiently far away from $0$, $\mathcal{I}_{0,t}(x)<|x_1|^{1/2n}$, which implies that 
\beq
\label{Borel-Cantelli}
\left|B_t\setminus\{x\in\upspace, \ x_2\le |x_1|^{1/n} \}\right|<\infty.
\eeq
Combining Theorem \ref{proposition_not_0} and \eqref{Borel-Cantelli}, 
\begin{corollary}
\label{corollary_not_0}
For any $t\ge 0$ and $B_t$ defined above, there must be some $x\in B_t$ such that $\sharm_{B_t}(x)>0$. 
\end{corollary}

By Theorem 1 of \cite{Stationary_Harmonic_Measure}, for any (infinite) $A\subset \upspace$, and any $x\in A$, $\sharm_A(x)\le C \sqrt{x_2}$ for some uniform constant $C<\infty$. Thus, from any configuration, the transition rates of $\bar\xi_{Ct}$ is always larger than the stationary harmonic measure. Thus, if one could define an infinite Stationary DLA model in $\upspace$ (we hope this question will be addressed in an ongoing study \cite{Stationary_DLA}), it will be dominated by $\bar\xi_{Ct}$ and Corollary \ref{corollary_not_0} shows that the infinite stationary DLA model starting from $L_0$, if well defined, will never hit an absorbing state and stop growing. The interesting thing here is, in order to show the model grows, we actually need to show it grows slowly. 

\section*{Acknowledgments} 
We would like to thank Itai Benjamini, Noam Berger, Marek Biskup and Gady Kozma for fruitful discussions. %\note{change references to bibtex}

\bibliography{upper_half_plane}
\bibliographystyle{plain}

\end{document}